\documentclass[11pt]{article}
\usepackage{amsmath, amsthm, amssymb}

\usepackage{hyperref}

\newtheorem{theorem}{Theorem}[section]
\newtheorem*{theorem*}{Theorem}
\newtheorem{proposition}[theorem]{Proposition}
\newtheorem{lemma}[theorem]{Lemma}
\newtheorem{problem}[theorem]{Problem}

\newtheorem{conjecture}[theorem]{Conjecture}
\newtheorem{example}[theorem]{Example}

\theoremstyle{remark}
\newtheorem{remark}{Remark}
\allowdisplaybreaks[4]
\numberwithin{equation}{section}
\date{}
\newcommand{\dist}{\mathrm{dist}}

\renewcommand{\div}{{\rm{\,div\,}}}

\newcommand{\loc}{{\rm{\,loc\,}}}
\newcommand{\supp}{{\rm{\,supp\,}}}

\renewcommand{\thefootnote}{\fnsymbol{footnote}}

\title{$L^p$ estimate for positive harmonic functions near singularities and B\^{o}cher type theorems}
\author{Shuimu Li$^{1}$}

\begin{document}

\footnotetext{1. School of Mathematical Sciences, CMA-Shanghai, Shanghai Jiao Tong University, Shanghai { \rm 200241}, China;}

\footnotetext{The first author Shuimu Li (lsm000524@sjtu.edu.cn) is partially supported by NSFC-12031012 and NSFC-11831003. }
\renewcommand{\thefootnote}{\arabic{footnote}}

\maketitle
\begin{abstract}
\noindent In this paper, positive solutions to the Laplace equation with 1-dimensional circular singularities are investigated. First, we establish $L^p$ integrability estimates for such solutions $u$ near the singularities, in comparison with classical $L^1$ estimates. Then we characterize $-\Delta u$ around the circle. The method we developed may also be adapted to higher dimensional as well as more generalized cases.  


\noindent{\bf {MSC}}:
                        35A21, 35B09, 35B65, 35J05.  
\end{abstract}

\tableofcontents
\section{Introduction}
The analysis of singular solutions to partial differential equations has caused widespread interest in both analysis and geometry. Removable singularities were first investigated by Riemann in the context of holomorphic functions on the punctured ball in $\mathbb{C}$. In PDE theory, the proposition is formulated as follows:
\begin{proposition}
If $-\Delta u=0$ on the punctured unit ball $B_1(0)\setminus\{0\}$ in $\mathbb{R}^n$ and $u(x)=o(|x|^{2-n})$ as $x\rightarrow0$, then $u$ extends to a harmonic function on $B_1(0)$.
\end{proposition}
Bochner \cite{bochner1956} generalized Riemann's theorem to arbitrary sets of singularity. In the most general setting, one may ask:
\begin{problem}
Suppose $M$ is a smooth manifold and $A\subset M$ is a subset of $M$. $L$ is a differential operator with $Lu=0$ on $M\setminus A$. What are the conditions on $u,A,L$ such that there exists an extension $\tilde{u}$ of $u$ on $M$ such that $\tilde{u}=u$ on $M\setminus A$ and that $L\tilde{u}=0$ on $M$?
\end{problem}
In a series of papers, Carleson \cite{carleson1963,carleson1967selected}, Serrin \cite{serrin1964local,Serrin1964RemovableSO}, Littman \cite{littman1967polar} and Harvey and Polking \cite{harveypolking} produced fruitful results under various circumstances, where $u$ belongs to $L^p$ or $C^\alpha$ class of functions and the size of $A$ is measured in terms of Minkowski content, Hausdorff dimension or analytic capacity. Also, much attention has been paid to nonlinear equations, including works by Nirenberg \cite{Nirenberg80removablesingularities}, Baras and Pierre \cite{baras1984singularites}, V\'{e}ron \cite{veron1996singularities} and Brezis and Nirenberg \cite{brezis1997removable}. 

Another approach which enjoys more popularity in geometry and complex analysis deals with positive singular solutions, starting from B\^{o}cher's theorem \cite{bocher1903singular}, in which the set of singularities cannot be removed:
\begin{theorem}[B\^{o}cher, 1903]
Suppose $u\geq0,u\in L^1_\loc(B_1\setminus\{0\})$ satisfies $-\Delta u=0$ in $B_1\setminus\{0\}$, then $u\in L^1_\loc(B_1)$ and that there exists some constant $a\geq0$ such that $-\Delta u=a\delta_0$ in $\mathcal{D}'(B_1)$.
\end{theorem}
Schoen and Yau \cite{Yau1988} studied positive singular harmonic functions on the limit set of the Kleinian group. Caffarelli, Gidas and Spruck \cite{caffarelli1989}, Li \cite{li1996local} analyzed the behavior of singular solutions to a nonlinear elliptic equation. Armstrong, Smart and Sirakov \cite{armstrong2011} investigated fundamental solutions to homogenous fully nonlinear elliptic equations. Li and Nguyen \cite{linguyen2014} generalized B\^{o}cher's theorem and Harnack inequalities to a class of conformally invariant fully nonlinear degenerate elliptic equations, closely related to the Yamabe problem. Li, Wu, Xu and Liu \cite{li2018maximum,li2020non} worked on B\^{o}cher type theorems for fractional operators and established some maximum principles. A comprehensive survey of isolated singularities in PDE theory can be found in Ghergu and Taliaferro \cite{ghergu_taliaferro_2016}. 

Although much has been understood in the behavior of solutions near isolated singularities, generalizations of B\^{o}cher's theorem in higher dimensional singularities remain unknown. In this paper, we focus on establishing B\^{o}cher type theorems for 1-dimensional circular singularities. To this end, a $L^p$ integrability estimate of positive harmonic functions near the singularities plays a crucial role in combining our effort with previous works summarized in Polking \cite{polking1984survey}, eventually entailing results of B\^{o}cher type:
\begin{theorem}\label{removal1}
Suppose $u\geq0,u\in L^1_\loc(B_2\setminus\Gamma)$ satisfies $-\Delta u=0$ in $B_2\setminus\Gamma$, where $B_2$ is the $n$-dimensional ball of radius $2$ and $\Gamma$ is the $1$-dimensional circle
\begin{equation}
    \Gamma=\{x\in\mathbb{R}^n:x_1^2+x_2^2=1,x_3=\cdots=x_n=0\}
\end{equation}embedded in $\mathbb{R}^n$. Then for any $1\leq p<\frac{n}{n-2},n\geq4$,
\begin{enumerate}
    \item $u\in L^p_\loc(B_2)$.
    \item There exists a distribution $v\in W^{-2,p}(\Gamma)$ on $\Gamma$ such that for all test function $\phi\in C_c^\infty(B_2)$,
    \begin{equation}\label{characterization2}
        \langle-\Delta u,\phi\rangle=\langle v,\tilde{\phi}\rangle,
    \end{equation}where $\tilde{\phi}=\phi|_\Gamma$ is the restriction of $\phi$ on $\Gamma$. 
\end{enumerate}
\end{theorem}

\section{$L^p$ estimates}

We shall prove that $u$ is in fact $L^p$ for some $p>1$ near the singularities, namely the first claim in Theorem \ref{removal1}.
\begin{proposition}
\label{Lp}
Suppose $u\geq0,u\in L^1_\loc(B_2\setminus\Gamma)$ satisfies $-\Delta u=0$ in $B_2\setminus\Gamma$, where $B_2$ is the $n$-dimensional ball of radius $2$ and $\Gamma$ is the $1$-dimensional circle
\begin{equation}
    \Gamma=\{x\in\mathbb{R}^n:x_1^2+x_2^2=1,x_3=\cdots=x_n=0\}
\end{equation}embedded in $\mathbb{R}^n$. Then for any $1\leq p<\frac{n}{n-2},n\geq4$, we have $u\in L^p_\loc(B_2)$.

\end{proposition}

\begin{proof}
\begin{enumerate}
    \item \textbf{Basic setup.}
    
    Denote by $\Gamma_r:=\{x\in\mathbb{R}^n:\dist(x,\Gamma)\leq r\}$ the $r$-neighbourhood of $\Gamma$. We shall utilize the following "canonical" parametrization of $\Gamma_{\frac{1}{2}}$:
    \begin{equation}
        \Phi(\rho,\theta_1,\theta_2,\cdots,\theta_{n-2},\varphi)=(x_1,x_2,\cdots,x_n),
    \end{equation}where
    \begin{equation}
        \begin{cases}
            x_1=(1+\rho\cos\theta_1\cdots\cos\theta_{n-2})\cos\varphi,\\
            x_2=(1+\rho\cos\theta_1\cdots\cos\theta_{n-2})\sin\varphi,\\
            x_3=\rho\cos\theta_1\cdots\cos\theta_{n-3}\sin\theta_{n-2},\\
            x_4=\rho\cos\theta_1\cdots\cos\theta_{n-4}\sin\theta_{n-3},\\
            \cdots,\\
            x_{n-1}=\rho\cos\theta_1\sin\theta_2,\\
            x_n=\rho\sin\theta_1,
        \end{cases}\quad\rho\in\left[0,\frac{1}{2}\right],\theta_1,\cdots,\theta_{n-2},\varphi\in[0,2\pi),
    \end{equation}so that the tubular neighbourhood $\Gamma_r=\{0\leq\rho\leq r\}$. Indeed, this is just the coordinates on the doubly warped product 
    \begin{equation}
        \Gamma_r=\bigsqcup_{q\in\Gamma}B_r^{n-1}(q)=I\times\mathbb{S}^{n-2}\times\Gamma,\quad 0<r<\frac{1}{2},
    \end{equation}where $B_r^{n-1}(q)$ stands for the $n-1$-dimensional ball centered at $q$ with radius $r$ on the normal plane $T_q^\perp\Gamma$, and $I$ being interval $[0,r]$. 
    The metric tensor is then given by
    \begin{equation}
    g=d\rho^2+\rho^2ds_{n-2}^2+f^2(\rho)d\varphi^2.
    \end{equation}Here $ds_{n-2}^2$ denotes the spherical metric on $\mathbb{S}^{n-2}$ and $f(\rho)=1+K(\boldsymbol{\theta})\rho$, where $K(\boldsymbol{\theta})=\cos\theta_1\cdots\cos\theta_{n-2}$ is independent of $\rho$, $|K|\leq1$.

   It follows that $\nabla=(\frac{\partial}{\partial\rho}\cdot)\textbf{e}_\rho+\nabla_{\tau}$ and the Laplacian 
    \begin{equation}\label{deltag}
        \Delta=\frac{1}{\rho^{n-2}(1+K\rho)}\frac{\partial}{\partial\rho}\left(\rho^{n-2}(1+K\rho)\frac{\partial}{\partial\rho}\cdot\right)+\Delta_\tau,
    \end{equation}$\nabla_\tau,\Delta_\tau$ consisting of derivatives independent of the variable $\rho$.
    
    \item \textbf{Cutoff function.}
    
    For fixed $\varepsilon>0$ sufficiently small and $\alpha\in(0,1)$, construct the barrier function 
    \begin{equation}
        \varphi_\varepsilon(x)=1-\left(\frac{\varepsilon}{\rho}\right)^\alpha, \quad x\in B_2\setminus\Gamma
    \end{equation}and let
    \begin{equation} 
        \psi_\varepsilon(x)=\max\{\varphi_\varepsilon(x),0\}=\begin{cases}
           \displaystyle 1-\left(\frac{\varepsilon}{\rho}\right)^\alpha,&x\in B_2\setminus\Gamma_\varepsilon,\\
           0,&x\in\Gamma_\varepsilon,
        \end{cases}
    \end{equation}
    so that $\psi_\varepsilon$ serves as a cutoff near $\Gamma$ with
    \begin{equation}
        \psi_\varepsilon\in C(B_2),\quad 0\leq\psi_\varepsilon\leq1,\quad \supp{\psi_\varepsilon}\subset B_2\setminus\Gamma_\varepsilon.
    \end{equation}
    
    Plus, a standard cutoff function $\eta$ is introduced to eradicate the technicalities near the boundary of $B_2$, namely $\eta\in C_c^\infty(B_2),0\leq\eta\leq1$ with 
    \begin{equation}
        \eta(x)=\begin{cases}
            1,&\rho\leq\frac{1}{2}r_0;\\
            0,&\rho>\frac{3}{4}r_0,
        \end{cases}
    \end{equation}
    where $2\varepsilon<r_0\ll\frac{1}{2}$ is independent of $\varepsilon$.     Note that the parameter $r_0$ in $\eta$ corresponds to that in Lemma
    \ref{barrierestimate}.
    
    \item \textbf{Gradient estimate.}

    Without loss of generality, assume $u\geq1$ in $B_2\setminus\Gamma$. We start from estimating a particular quantity, denoted by \begin{equation}
    I_\varepsilon=\int_{B_2}\frac{\rho^{\gamma-1}\frac{\partial u}{\partial\rho}}{u^{\beta-1}}\eta\psi_\varepsilon,
\end{equation}
where $\gamma\in(0,1)$ and $\beta\in(1,2)$ are yet to be chosen. Also denote
\begin{equation}
    A_\varepsilon=\left(\int_{B_2}\frac{\rho^\gamma|\nabla u|^2}{u^\beta}\eta\psi_\varepsilon\right)^{\frac{1}{2}},\quad B_\varepsilon=\left(\int_{B_2}\frac{u^{2-\beta}}{\rho^{2-\gamma}}\eta\psi_\varepsilon\right)^\frac{1}{2}.
\end{equation}Clearly by H\"{o}lder's inequality, 
\begin{equation}\label{calculation204}
    |I_\varepsilon|\leq A_\varepsilon B_\varepsilon.
\end{equation}The aim is to control $A_\varepsilon^2$ and $B_\varepsilon^2$ by $|I_\varepsilon|$ so that all three integrals are bounded by constants independent of $\varepsilon$. 
\begin{enumerate}
    \item Step 1.
    
    Since all entries are smooth and bounded on $B_2\setminus\Gamma_\varepsilon$, the divergence theorem implies
\begin{equation}\label{calculation103}
\begin{aligned}
    \gamma(2-\beta)I_\varepsilon&=\int_{B_2}\eta\psi_\varepsilon\nabla u^{2-\beta}\cdot\nabla\rho^\gamma
    =-\int_{B_2}u^{2-\beta}\div(\eta\psi_\varepsilon\nabla\rho^\gamma)\\
    &=-\int_{B_2}u^{2-\beta}\eta\psi_\varepsilon\Delta\rho^\gamma-\int_{B_2}u^{2-\beta}\eta\nabla\psi_\varepsilon\cdot\nabla\rho^\gamma-\int_{B_2}u^{2-\beta}\psi_\varepsilon\nabla\eta\cdot\nabla\rho^\gamma\\
    &:=-I_1-I_2-I_3.
\end{aligned}
\end{equation}

Using \eqref{deltag} we readily compute $\Delta\rho^\gamma\geq\gamma(n+\gamma-3-s)\rho^{\gamma-2}>0$, where $s=\frac{r_0}{1-r_0}$ is sufficiently small provided that $r_0\ll\frac{1}{2}$. Hence
\begin{equation}
I_1=\int_{B_2}u^{2-\beta}\eta\psi_\varepsilon\Delta\rho^\gamma\geq\gamma(n+\gamma-3-s)\int_{B_2}\frac{u^{2-\beta}}{\rho^{2-\gamma}}\eta\psi_\varepsilon,    
\end{equation}and since $\nabla\psi_\varepsilon\cdot\nabla\rho^\gamma=\gamma\alpha\varepsilon^\alpha\rho^{-\alpha+\gamma-2}>0$, we deduce $I_2>0$. Furthermore, since $\nabla\eta$ is supported away from $\Gamma$ uniformly in $\varepsilon$, clearly $|I_3|\leq C$ for some constant $C>0$ independent of $\varepsilon$.

Therefore, we may rewrite \eqref{calculation103} as
\begin{equation}
    \gamma(2-\beta)|I_\varepsilon|=|I_1+I_2+I_3|\geq |I_1|+|I_2|-|I_3|\geq\gamma(n+\gamma-3-s)B_\varepsilon^2-C,
\end{equation}namely
\begin{equation}\label{calculation104}
    \frac{c_0}{2-\beta}B_\varepsilon^2\leq|I_\varepsilon|+C,
\end{equation}where $c_0=n+\gamma-3-s>0$.

\item Step 2.

Likewise, since $u$ is harmonic on $B_2\setminus\Gamma_\varepsilon$, we also have
\begin{equation}\label{calculation202}
\begin{aligned}
    \gamma I_\varepsilon&=\int_{B_2}\frac{\nabla u\cdot\nabla\rho^\gamma}{u^{\beta-1}}\eta\psi_\varepsilon=-\int_{B_2}\rho^\gamma\div(\eta\psi_\varepsilon u^{1-\beta}\nabla u)\\
    &=(\beta-1)\int_{B_2}\frac{\rho^\gamma|\nabla u|^2}{u^\beta}\eta\psi_\varepsilon-\int_{B_2}\rho^\gamma\eta u^{1-\beta}\nabla\psi_\varepsilon\cdot\nabla u-\int_{B_2}\rho^\gamma\psi_\varepsilon u^{1-\beta}\nabla\eta\cdot\nabla u\\
    &:=J_1-J_2-J_3
\end{aligned}
\end{equation}
Similarly as before, $|J_3|\leq C$ for some constant $C>0$ independent of $\varepsilon$. It suffices to estimate $J_2$. We need the following lemma:

\begin{lemma}\label{barrierestimate}
There exists $r_0>0$ such that for any fixed $\varepsilon>0$ sufficiently small, 
\begin{equation}\label{calculation201}
    -\Delta\psi_\varepsilon+\frac{\rho^\gamma\nabla u}{u^\beta}\cdot\nabla\psi_\varepsilon\leq0
\end{equation}in $\mathcal{D}'(\Gamma_{r_0}\setminus\Gamma)$. 
\end{lemma}
\begin{proof}
Using \eqref{deltag}, we readily compute that for $n\geq4$,
    \begin{equation}\label{deltabarrier}
        -\Delta\varphi_\varepsilon=\frac{\alpha(\alpha+3-n)\varepsilon^\alpha}{\rho^{\alpha+2}(1+K\rho)}\left(1+\frac{n-\alpha-2}{n-\alpha-3}K\rho\right)
        \sim O\left(-\frac{\varepsilon^\alpha}{\rho^{\alpha+2}}\right)
    \end{equation}in $\Gamma_{r_0}\setminus\Gamma$ for some fixed $2\varepsilon<r_0\ll\frac{1}{2}$ independent of $\varepsilon$. 
    
Since $u$ is harmonic in any $\Gamma_{2r}\setminus\Gamma_{\frac{r}{2}}$, the differential Harnack inequality\cite{han2011elliptic} implies
\begin{equation}
    \frac{|\nabla u|}{u}(x)\leq\frac{C}{\dist(x,\Gamma)},\quad\forall x\in \Gamma_{\frac{1}{2}}\setminus\Gamma,
\end{equation}for some $C=C(n)>0$ independent of $\varepsilon,x$. Hence we have $|\rho^\gamma u^{-\beta}\nabla u|\leq C\rho^{\gamma-1}$. Combining 
    $\frac{\partial\varphi_\varepsilon}{\partial\rho}=\alpha\varepsilon^\alpha\rho^{-\alpha-1}$ and \eqref{deltabarrier} yields
\begin{equation}
    -\Delta\varphi_\varepsilon(x)+\frac{\rho^\gamma\nabla u}{u^\beta}\cdot\nabla\varphi_\varepsilon(x)\leq-\varepsilon^\alpha(C_1\rho^{-\alpha-2}-C_2\rho^{-\alpha-2+\gamma})\leq0
\end{equation}provided that $\rho\leq r_0$ which is sufficiently small while still independent of $\varepsilon$. We conclude by noting the fact that the maximum of two subharmonic functions is also subharmonic; see Lemma 5.1 in \cite{li2020non} for a detailed proof. In fact, \eqref{calculation201} holds pointwise on $\Gamma_{r_0}\setminus\Gamma$ except at $\rho=\varepsilon$, where $\psi_\varepsilon$ happens to be not twice differentiable.
\end{proof}

Now that we have \eqref{calculation201}, 
\begin{equation}
\begin{aligned}
    J_2&=\int_{B_2}\rho^\gamma\eta u^{1-\beta}\nabla\psi_\varepsilon\cdot\nabla u=\int_{B_2}u\eta\left(-\Delta\psi_\varepsilon+\frac{\rho^\gamma\nabla u}{u^\beta}\cdot\nabla\psi_\varepsilon\right)-u\eta(-\Delta)\psi_\varepsilon\\
    &
    \leq\langle u\eta,\Delta\psi_\varepsilon\rangle=\langle\Delta(u\eta),\psi_\varepsilon\rangle=\int_{B_2}\psi_\varepsilon(2\nabla u\cdot\nabla\eta+u\Delta\eta)\leq C,
\end{aligned}  
\end{equation}where we make use of the fact that the derivatives of $\eta$ are supported away from $\Gamma_{r_0}$. Hence substituting into \eqref{calculation202} yields
\begin{equation}
    \gamma I_\varepsilon=(\beta-1)A_\varepsilon^2-J_2-J_3\geq(\beta-1)A_\varepsilon^2-C,
\end{equation}namely
\begin{equation}\label{calculation203}
    \frac{\beta-1}{\gamma}A_\varepsilon^2\leq|I_\varepsilon|+C.
\end{equation}

\item Step 3.

Combining \eqref{calculation204}, \eqref{calculation104} and \eqref{calculation203} yields

\begin{equation}
    \begin{cases}
        \displaystyle\frac{\beta-1}{\gamma}A_\varepsilon^2\leq A_\varepsilon B_\varepsilon+C,\\[2.5ex]
        \displaystyle\frac{c_0}{2-\beta}B_\varepsilon^2\leq A_\varepsilon B_\varepsilon+C.
    \end{cases}
\end{equation}Thus we have
\begin{equation}
    \left(\frac{\beta-1}{\gamma}-\frac{2-\beta}{c_0}\right)A_\varepsilon^2\leq C
\end{equation}for some constant $C>0$ independent of $\varepsilon$, as a consequence of the absorbing inequality. Since $\psi_\varepsilon\nearrow1$ as $\varepsilon\rightarrow0^+$, the monotone convergence theorem implies the estimate
\begin{equation}
    \int_{\Gamma_{\frac{1}{2}r_0}}\frac{\rho^\gamma|\nabla u|^2}{u^\beta}\leq\int_{B_2}\frac{\rho^\gamma|\nabla u|^2}{u^\beta}\eta\leq C<\infty 
\end{equation}holds for any $\beta\in(\frac{c_0+2\gamma}{c_0+\gamma},2)$.

It follows from the generalized H\"{o}lder's inequality and the Sobolev embedding that
\begin{equation}\label{calculation10003}
    C\geq\left(\int\frac{\rho^\gamma|\nabla u|^2}{u^\beta}\right)^t\int\rho^{-\gamma t}\geq\left(\int|\nabla u^{\frac{2-\beta}{2}}|^{\frac{2t}{t+1}}\right)^{t+1}\geq C\|u\|_{L^q}^{t(2-\beta)},
\end{equation}where $t<\frac{n-1}{\gamma}$ and $q=\frac{(2-\beta)nt}{(n-2)t+n}$. Above all, we have shown that $u\in L_\loc^q(B_2)$ for any
\begin{equation}\label{calculation10004}
    q=\frac{(2-\beta)nt}{(n-2)t+n}<\frac{c_0n(n-1)}{(c_0+\gamma)(n^2+(\gamma-3)n+2)}\nearrow\frac{n}{n-2},\quad\gamma\rightarrow0^+.
\end{equation}

\end{enumerate}

\end{enumerate}
\end{proof}

\begin{remark}
Traditional proofs of B\^{o}cher's theorem are based on $L^1$ estimates near isolated singularities, which are easier to establish, and $L^p$ integrability results eventually appear as consequences of the characterization of $u$, by virtue of the fundamental solution; see for instance \cite{li2020non}. The approach we take here is thus relatively new. 
\end{remark}

\begin{remark}
The $L^p,1\leq p<\frac{n}{n-2}$ regularity is sharp in the following sense: 
\begin{example}
Consider $u(x)=\Phi_{x_0}=c|x-x_0|^{2-n}>0$, i.e. the fundamental solution centered at some $x_0\in\Gamma$. Clearly $-\Delta u=\delta_{x_0}$ in $\mathcal{D}'(\mathbb{R}^n)$ so that $-\Delta u=0$ pointwise in $\mathbb{R}^n\setminus\Gamma$ and $u\in L^1_\loc(B_2\setminus\Gamma)$. But obviously $u\notin L_\loc^p(B_2)$ for any $p\geq\frac{n}{n-2}$.
\end{example}
\end{remark}

\section{Generalization of B\^{o}cher's theorem}

With Proposition \ref{Lp} at hand, we are able to characterize $-\Delta u$ around the circular singularities $\Gamma$ and complete the proof of Theorem \ref{removal1}.

\begin{theorem}
Under the hypotheses in Proposition \ref{Lp}, there exists a distribution $v\in W^{-2,p}(\Gamma)$ on $\Gamma$ such that for all test function $\phi\in C_c^\infty(B_2)$,
    \begin{equation}
        \langle-\Delta u,\phi\rangle=\langle v,\tilde{\phi}\rangle,
    \end{equation}where $\tilde{\phi}=\phi|_\Gamma$ is the restriction of $\phi$ on $\Gamma$. 
\end{theorem}
\begin{proof}
The proof we present here is rather detailed and specific; see Theorem 6.1 in \cite{harveypolking} for more general cases. 

Use polar coordinates on the $x_1Ox_2$ plane to reparametrize $B_2$:
\begin{equation}
    \Phi(r,\theta,x_3,\cdots,x_n)=(r\cos\theta,r\sin\theta,x_3,\cdots,x_n),\quad\theta\in[0,2\pi),r,x_3,\cdots,x_n\in[0,2].
\end{equation}
Note that $\Gamma=\{r=1,x_3=\cdots=x_n=0\}$. For any test function $\phi\in C_c^\infty(B_2)$, denote 
\begin{equation}
    \psi(x)=\psi(r,\theta,x_3,\cdots,x_n)=\phi(r,\theta,x_3,\cdots,x_n)-\phi(1,\theta,0,\cdots,0)\zeta,
\end{equation}where $\zeta=\zeta(r,x_3,\cdots,x_n)\in C_c^\infty(B_2)$ is a smooth cutoff function with $0\leq\zeta\leq1$ and $\zeta\equiv1$ in $\Gamma_{\frac{1}{2}}$. Hence by construction, we deduce 
\begin{equation}
    \psi|_\Gamma=\psi(1,\theta,0,\cdots,0)\equiv0,\quad\psi\in C_c^\infty(B_2),\quad|\psi(x)|\leq C\dist(x,\Gamma),\;\;\forall x\in\Gamma_{\frac{1}{2}}.
\end{equation}

For any $\varepsilon>0$, consider a family of functions $\varphi_\varepsilon\in C_c^\infty(B_2)$ such that
\begin{equation}
   \varphi_\varepsilon(x)=\begin{cases}
    1,&x\in\Gamma_{\frac{1}{2}\varepsilon},\\
    0,&x\in B_2\setminus\Gamma_\varepsilon
    \end{cases}
\end{equation}and that for all $j\in\mathbb{N}$, there exists some constant $C_j>0$ independent of $\varepsilon$ such that 
\begin{equation}
    |\nabla^j\varphi_\varepsilon|\leq C_j\varepsilon^{-j}.
\end{equation}
Thus we have
\begin{equation}
\begin{aligned}
     |\langle-\Delta u,\psi\rangle|&\leq|\langle-\Delta u,(1-\varphi_\varepsilon)\psi\rangle|+|\langle-\Delta u,\varphi_\varepsilon\psi\rangle|=|\langle-\Delta u,\varphi_\varepsilon\psi\rangle|=|\langle u,-\Delta(\varphi_\varepsilon\psi)\rangle|\\
     &\leq\|u\chi_{\Gamma_\varepsilon}\|_{L^1}\|-\Delta(\varphi_\varepsilon\psi)\|_{L^\infty}\leq\|u\|_{L^p(\Gamma_{\frac{1}{2}})}\|\chi_{\Gamma_\varepsilon}\|_{L^{p'}}\|-\Delta(\varphi_\varepsilon\psi)\|_{L^\infty}.
\end{aligned}
\end{equation}Since $\|\chi_{\Gamma_\varepsilon}\|_{L^{p'}}=({\rm{vol}}(\Gamma_\varepsilon))^{\frac{1}{p'}}\leq C\varepsilon^{\frac{n-1}{p'}}$ and 
\begin{equation}
    |-\Delta(\varphi_\varepsilon\psi)|\leq|\psi(-\Delta)\varphi_\varepsilon|+2|\nabla\varphi_\varepsilon\cdot\nabla\psi|+|\varphi_\varepsilon(-\Delta)\psi|\leq C\varepsilon\cdot\varepsilon^{-2}+2C\varepsilon^{-1}+C\leq C\varepsilon^{-1},
\end{equation}we have
\begin{equation}
    |\langle-\Delta u,\psi\rangle|\leq C\varepsilon^{\frac{n-1}{p'}-1}\leq C\varepsilon^{1-\frac{2}{n}}\leq C\varepsilon^{\frac{1}{2}}.
\end{equation}Sending $\varepsilon\rightarrow0$ implies $\langle-\Delta u,\psi\rangle=0$. Therefore, for any test function $\phi\in C_c^\infty(B_2)$,
\begin{equation}
    \langle-\Delta u,\phi\rangle=\langle-\Delta u,\phi(1,\theta,0,\cdots,0)\zeta\rangle=\int_\Gamma w\frac{\partial^2\phi}{\partial\tau^2}dl:=\langle v,\phi|_\Gamma\rangle
\end{equation}for some $w\in L^p(\Gamma)$ and distribution $v=w''\in W^{-2,p}(\Gamma)$.
\end{proof}

\begin{remark}
Unlike the traditional proofs of B\^{o}cher's theorem, the scheme we follow here focuses on an apriori $L^p$ regularity estimate near singularities, subsequently producing B\^{o}cher type results. It is worth noting that the $L^p,p<\frac{n}{n-2}$ regularity alone could entail such results, namely that in \eqref{characterization2} when the action of $-\Delta u$ on some test function $\phi\in C_c^\infty(B_2)$ does not involve derivatives of $\phi$ in the normal direction with respect to the set of singularities $\Gamma$. 

Thus it occurs naturally that the positivity of $u$ here might be of help to annihilate tangential directional derivatives when $-\Delta u$ acts on test functions, or equivalently, improve the regularity of $v\in W^{-2,p}(\Gamma)$ in \eqref{characterization2} to $W^{-1,p}$ generalized functions or even distributions of order 0. Indeed, one is tempted to propose the following conjecture: 
\begin{conjecture}
The distribution $v$ in Theorem \ref{removal1} is positive, namely there exists a positive Radon measure $\mu$ on $\Gamma$ such that for any test function $\phi\in C_c^\infty(B_2),$
\begin{equation}
    \langle-\Delta u,\phi\rangle=\int_\Gamma\tilde{\phi}d\mu,
\end{equation}where $\tilde{\phi}=\phi|_\Gamma$ is the restriction of $\phi$ on $\Gamma$.
\end{conjecture}
\end{remark}

\bibliographystyle{siam}
\bibliography{ref}
\end{document}